\documentclass[12pt,a4paper]{amsart}
\usepackage{amssymb,amsmath}
\usepackage{mathrsfs}

\headsep=1cm \numberwithin{equation}{section}
\newtheorem{theorem}{Theorem}[section]
\newtheorem{lemma}[theorem]{Lemma}
\newtheorem{proposition}[theorem]{Proposition}

\theoremstyle{definition}

\newcommand\Supp{\operatorname{Supp}}
\newcommand\Ass{\operatorname{Ass}}

\newcommand\Ann{\operatorname{Ann}}

\newcommand\Rad{\operatorname{Rad}}

\newcommand\grade{\operatorname{grade}}
\newcommand\height{\operatorname{height}}
\newcommand\mAss{\operatorname{mAss}}
\begin{document}

\title[Linearly equivalent ideal topologies]{Note on linearly equivalent ideal topologies over Noetherian modules}%
\author{Adeleh Azari, Simin Mollamahmoudi and Reza Naghipour$^*$}%
\address{Department of Mathematics, University of Tabriz,
Tabriz, Iran, and School of Mathematics, Institute for Research in Fundamental
Sciences (IPM), P.O. Box: 19395-5746, Tehran, Iran.}%
\email{adeleh\_azari@yahoo.com (Adeleh Azari)}
\email{mahmoudi.simin@yahoo.com (Simin Mollamahmoudi)}
\email{naghipour@ipm.ir (Reza Naghipour)} \email {naghipour@tabrizu.ac.ir (Reza Naghipour)}
\thanks{ 
This research was in part supported by a grant from IPM \\
$^*$Corresponding author: e-mail: {\it naghipour@ipm.ir} (Reza
Naghipour)}%


\begin{abstract}
Let $R$ be a commutative Noetherian ring, and let $N$ be a  non-zero finitely generated $R$-module. In this paper,
the main result asserts that for any  $N$-proper ideal $\frak a$  of  $R,$  the $\frak a$-symbolic topology on $N$ is linearly equivalent  to
the  $\frak a$-adic topology on $N$ if and only if, for every $\frak p\in \Supp(N)$,
$\Ass_{R_{\mathfrak  {p} }}N_{\mathfrak  {p}}$ consists of a single prime ideal and $\dim N\leq 1$.
\end{abstract}
\maketitle
\section {Introduction}
Let $R$ be a commutative Noetherian ring, $\frak a$ an ideal of $R$  and let $N$ be a  non-zero finitely generated $R$-module.  For a non-negative integer $n$,
the $n$th {\it symbolic power} of $\frak a$ w.r.t. $N$, denoted by $(\frak aN)^{(n)}$, is defined to be the intersection of those primary
components of $\frak a ^nN$ which correspond to the minimal elements of $\Ass_RN/\frak a N$.  Then the $\frak a$-adic filtration $\{\frak a ^nN\}_{n\geq0}$
and the $\frak a$-symbolic filtration $\{(\frak aN)^{(n)}\}_{n\geq0}$ induce topologies on $N$ which are called the  $\frak a$-adic topology and $\frak a$-symbolic topology, respectively.  These two topologies are said to be linearly equivalent if, there is an integer $k\geq0$ such that
$(\frak aN)^{(n+k)}\subseteq \frak a ^nN$ for all integers $n\geq0$.

Our main point of the present paper concerns an invistigation of the linearly equivalent of the $\frak a$-symbolic and the  $\frak a$-adic topology topologies on $N$. More precisely we shall show that:

\begin{theorem}
Let $R$ be a commutative Noetherian ring, and let $N$ be a  non-zero finitely generated $R$-module. Then for any  $N$-proper ideal $\frak a$  of  $R,$  the $\frak a$-symbolic topology on $N$ is linearly equivalent  to
the  $\frak a$-adic topology on $N$ if and only if, for every $\frak p\in \Supp(N)$,
$\Ass_{R_{\mathfrak  {p} }}N_{\mathfrak  {p}}$ consists of a single prime ideal and $\dim N\leq 1$.
\end{theorem}

The result in Theorem 1.1 is proved in Theorem 2.4. Our method is based on the theory of the asymptotic and essential primes
of $\frak a$ w.r.t. $N$ which were introduced by McAdam \cite{Mc2}, and in \cite{ah}, Ahn extended these concepts to a finitely generated $R$-module $N$.  One of our tools for proving Theorem 1.1 is the following, which plays a key role in this paper.

\begin{proposition} \label{1.4}
Let $R$ be a commutative Noetherian ring and $\frak{a}$  an ideal of $R$.  Let $N$ be a non-zero finitely generated $R$-module  such that $\dim N>0,$
and let $\mathfrak  {p}\in \Supp(N)\cap V(\frak{a}).$  Then the following conditions are equivalent:

$\rm(i)$  $\mathfrak  {p}\in\bar{A^{*}}(\frak{a},N).$

$\rm(ii)$ $\mathfrak{p}\in\bar{A^{*}}(\frak{a}\frak{b},N)$,  for any $N$-proper ideal $\frak{b}$ of $R$ with $\height_N \frak{b} > 0$.

$\rm(iii)$ $\mathfrak{p}\in\bar{A^{*}}(x\frak{a},N),$  for any  $N$-proper element $x$ of $R$ with
  $x\not\in{\bigcup_{\mathfrak  {p}\in \mAss_{R}N}}\mathfrak  {p}.$
\end{proposition}

A prime ideal $\frak{p}$ of $R$ is called a {\it quitessential} (resp. {\it quitasymptotic}) {\it prime ideal of }$\frak{a}$
w.r.t. $N$ precisely when there exists $\frak{q}\in \Ass_{R^*_\frak{p}}N^*_\frak{p}$ (resp. $\frak{q}\in
\mAss_{R^*_\frak{p}}N^*_\frak{p}$) such that $\Rad(\frak{a}R^*_\frak{p}+ \frak{q})= \frak{p}R^*_\frak{p}$. The
set of quitessential (resp. quitasymptotic) prime ideals of $\frak{a}$ w.r.t. $N$ is denoted by $Q(\frak{a},N)$ (resp. $\bar{Q^*}(\frak{a}, N)$)
which is a finite set.

We denote by $\mathscr{R}$  the {\it graded Rees ring} $R[u,\frak{a}t] :=\oplus_{n\in \mathbb{Z}}\frak{a}^nt^n$ of $R$ w.r.t.
$\frak{a}$, where $t$ is an indeterminate and $u= t^{-1}$. Also, the {\it graded Rees module} $N[u, \frak{a}t] := \oplus_{n\in
\mathbb{Z}}\frak{a}^nN$ over $\mathscr{R}$ is denoted by $\mathscr{N}$, which is a finitely generated graded
$\mathscr{R}$-module.  Then we say that a  prime ideal $\mathfrak  {p}$ of $R$ is  an {\it essential prime ideal
of}  $\frak a$  w.r.t. $N,$ if $\mathfrak  {p}=\mathfrak  {q}\cap R$
for some $\mathfrak  {q} \in Q(u\mathscr{R},\mathscr{N}).$ The set of essential prime ideals
of $\frak{a}$ w.r.t. $N$ will be denoted by $E(\frak{a},N).$

Also, the {\it asymptotic prime ideals of} $\frak{a}$ w.r.t. $N,$ denoted by $\bar{A^{*}}(\frak{a},N)$, is defined to be the set
$\{\frak{q}\cap R\mid\, \frak{q}\in \bar{Q^{*}}(u\mathscr{R},\mathscr{N}) \}$.

In \cite{STY},  Sharp et al. introduced the concept  of integral closure of $\frak{a}$ relative to $N$, and they showed that
 this concept have properties which reflect some of those of the  usual concept of integral closure introduced by Northcott and
 Rees in \cite{NR}. The integral  closure of $\frak{a}$ relative to  $N$ is denoted  by $\frak{a}^{-(N)}$.
In \cite{NS2}, it is shown that the  sequence $\{\Ass_RR/(\frak{a}^n)^{-(N)}\}_{n\geq1}$, of associated prime
 ideals, is increasing and ultimately constant; we  denote its ultimate constant value by $\hat{A^*}(\frak{a}, N)$.  In the case
 $N= R$, $\hat{A^*}(\frak{a}, N)$ is the asymptotic primes $\hat{A^*}(\frak{a})$
 of $\frak{a}$ introduced by Ratliff in  \cite{ra1}. Also, it is shown in \cite[Proposition 3.2]{na2} that
 $\hat{A^*}(\frak{a}, N)=\bar{A^{*}}(\frak{a},N)$.

If $(R,\frak{m})$ is local, then $R^*$ (resp. $N^*$) denotes the completion of $R$ (resp. $N$) w.r.t. the
$\frak{m}$-adic topology. In particular, for every prime ideal $\frak{p}$ of $R$, we denote $R^*_\frak{p}$ and $N^*_\frak{p}$ the
$\frak{p}R_\frak{p}$-adic completion of $R_\frak{p}$ and $N_\frak{p}$, respectively.  For any ideal $\frak{b}$ of $R$, {\it
the radical of} $\frak{b}$, denoted by $\Rad(\frak{b})$, is defined to be the set $\{x\in R \,: \, x^n \in \frak{b}$ for some $n \in
\mathbb{N}\}$. Finally, for each $R$-module $L$, we denote by $\mAss_RL$ the set of minimal prime ideals of $\Ass_RL$.

 Recall that an ideal $\frak b$ of $R$ is called
$N$-proper if $N/\frak bN\neq0$, and, when this the case, we define the $N$-{\it height} of $\frak b$ (written  $\height_N \frak{b}$) to be
$$\inf \lbrace \height_N\frak{p}: \frak{p}\in\Supp{N}\cap V(\frak b)\rbrace,$$ where $ \height_N\frak{p}$ is
defined to be the supremum of lengths of chains of prime ideals of $\Supp(N)$ terminating with $\frak{p}.$ Also, we say that an element
$x$ of $R$ is an $N$-proper element if $N/ xN\neq0$. For any unexplained notation and terminology we refer the reader  to \cite{BH} or \cite{N}.

\section{The main result}
Let $R$ be a commutative Noetherian ring,  and let $N$ be a  non-zero finitely generated $R$-module.  The purpose of the present paper is to  give an investigation of the linearly equivalent of the $\frak a$-symbolic and the  $\frak a$-adic topology topologies on $N$.The main goal of this section is Theorem 2.4.
The following proposition plays a key role in the proof of the main theorem.

\begin{proposition} \label{1.4}
Let $\frak{a}$ be an ideal of $R$ and let $N$ be a non-zero finitely generated $R$-module with $\dim N>0.$
 Let $\mathfrak  {p}\in \Supp(N)\cap V(\frak{a}).$  Then the following conditions are equivalent:

$\rm(i)$  $\mathfrak  {p}\in\bar{A^{*}}(\frak{a},N).$

$\rm(ii)$ $\mathfrak{p}\in\bar{A^{*}}(\frak{a}\frak{b},N)$,  for any $N$-proper ideal $\frak{b}$ of $R$ with $\height_N \frak{b} > 0$.

$\rm(iii)$ $\mathfrak{p}\in\bar{A^{*}}(x\frak{a},N),$  for any  $N$-proper element $x$ of $R$ with
  $x\not\in{\bigcup_{\mathfrak  {p}\in \mAss_{R}N}}\mathfrak  {p}.$

$\rm(iv)$  $\mathfrak{p}\in\bar{A^{*}}(x\frak{a},N),$ for some  $N$-proper element $x$ of $R$ with
  $x\not\in{\bigcup_{\mathfrak  {p}\in \mAss_{R}N}}\mathfrak  {p}.$

\end{proposition}
\begin{proof}
 (i)$\Longrightarrow$(ii): Let
$\mathfrak{p}\in\bar{A^{*}}(\frak{a},N)$
and let $\frak{b}$ be an $N$-proper ideal of $R$ such that $\height_N \frak{b} >0$. Then, in view of
\cite[Remark 2.4]{na2},
$$\mathfrak  {p}/\Ann_R N \in \hat{A^{*}}(\frak{a}+\Ann_R N/\Ann_R N). $$
Hence, as by
 \cite[Theorem 2.1]{na1},
$$\height_N \frak{b}=\height(\frak{b}+\Ann_R N/\Ann_R N) >0,$$
it follows from
 \cite[Proposition 3.26]{mc1}  that
$$\mathfrak  {p}/\Ann_R N \in \hat{A^{*}}(\frak{a}\frak{b}+\Ann_R N/\Ann_R N).$$
Therefore by using \cite[Remark 2.4]{na2},
we obtain that
$\mathfrak  {p}\in\bar{A^{*}}(\frak{a}\frak{b},N),$ as required.

 (ii)$\Longrightarrow$(iii): Let (ii) hold and let
$x$ be an $N$-proper element of $R$ such that
$x\not\in{\bigcup_{\mathfrak  {p}\in \mAss_{R}N}}\mathfrak  {p}.$
Then it is easy to see that
$\height_N xR >0,$ and so according to the assumption (ii), we have
$\mathfrak  {p}\in\bar{A^{*}}(x\frak{a},N).$

(iii)$  \Longrightarrow $(iv):  Since $\dim N>0,$ there exists
$\mathfrak {q}\in \Supp N$ such that
$\height_N \mathfrak{q}>0.$ Hence
$\mathfrak{q}\nsubseteq{\bigcup_{\mathfrak  {p}\in  \mAss_{R}N}}\mathfrak  {p},$
and so there is
$x\in\mathfrak{q}$ such that
$x\not\in{\bigcup_{\mathfrak  {p}\in  \mAss_{R}N}}\mathfrak  {p}.$
 Consequently, it follows from the hypothesis (iii) that
$\mathfrak  {p}\in\bar{A^{*}}(x\frak{a},N).$

 (iv)$ \Longrightarrow$(i):  Let $x$ be an $N$-proper element of $R$ such that
$x\not\in{\bigcup_{\mathfrak  {p}\in  \mAss_{R}N}}\mathfrak  {p}$ and  let
 $\mathfrak{p}\in\bar{A^{*}}(x\frak{a},N).$  Then
$$\mathfrak {p}/\Ann_R N \in\hat{A^{*}}(x\frak{a}+\Ann_R N/\Ann_R N),$$
by \cite[Remark 2.4]{na2}.
Now, since
$x\not\in{\bigcup_{\mathfrak  {p}\in \mAss_{R}N}}\mathfrak  {p},$
it is easy to see that
$x+\Ann_R N$ is not in any minimal prime
$R/\Ann_R N.$
Therefore, it follows from  \cite[Proposition 3.26]{mc1} that
$$\mathfrak {p}/\Ann_R N \in \hat{A^{*}}(\frak{a}+\Ann_R N/\Ann_R N).$$
Consequently, in view of \cite[Remark 2.4]{na2},
$\mathfrak{p}\in\bar{A^{*}}(\frak{a},N),$ and this completes the proof.
\end{proof}

Before we state Theorem 2.4 which is our main result, we give a couple of lemmas that will be used in the proof of Theorem 2.4.
\begin{lemma} \label{1.5}
Let $(R,\mathfrak{m})$ be a local ring and let $N$ be a non-zero finitely
generated $R$-module such that $\dim N>0$ and that
$\Ass_R N$ has at least two elements. Then there is an ideal $\frak a$ of $R$ such that $\mathfrak{m}\in 	Q(\frak a,N)\setminus \mAss N/\frak aN.$
\end{lemma}
\begin{proof}
See
\cite[Proposition 4.2]{amn1}.
\end{proof}

\begin{lemma}\label{1.7}
 Let $N$ be  a non-zero finitely
generated $R$-module and let $\frak a$ be an
$N$-proper ideal of $R$. Then  $E(\frak a,N)=\mAss_{R}N/ \frak aN  $
if and only if the $\frak a$-symbolic topology is linearly equivalent to the $\frak a$-adic  topology.
\end{lemma}
\begin{proof}
The assertion follows easily from \cite[Theorem 4.1]{na3}.
\end{proof}

We are now ready to state and prove the main theorem of this paper
which is a characterization of the certain modules in terms of the linear
equivalence of certain topologies induced by families of submodules of
a  finitely generated  module $N$ over a commutative Noetherian ring $R$. We denote by $Z_R(N)$  the set of zero divisors on $N$, i.e.,
$Z_R(N):=\{r\in R\,|\, rx=0\, \text{for some}\, \, x(\neq 0) \in N  \}$.

\begin{theorem}\label{1.8}
 Let $N$ be  a non-zero finitely generated $R$-module. Then the following conditions are equivalent:

$\rm(i)$  For every $N$-proper ideal $\frak b$  of  $R,$  the $\frak b$-symbolic topology is linearly equivalent  to the  $\frak b$-adic topology.

$\rm(ii)$ $\dim N\leq 1$ and  $\Ass_{R_{\mathfrak  {p} }}N_{\mathfrak  {p}}$ consists of a single prime ideal, for all  $\mathfrak{p}\in\Supp(N).$
\end{theorem}
\begin{proof}
Suppose that (i) holds. Firstly, we show that
$\dim N\leq 1.$ To achieve this, suppose the contrary is true. That is
$\dim N>1. $ Then there exists
 $\mathfrak  {p}\in \Supp(N)$ such that
 $ \height_{N}\mathfrak{p}>1.$ Hence
 $ \mathfrak {p}\nsubseteq{\bigcup_{\mathfrak{q}\in \mAss_{R}N}} \mathfrak{q},$
 and so there exists
 $x\in \mathfrak {p}$ such that
 $x\not\in \bigcup_{\mathfrak{q}\in \mAss_{R}N}\mathfrak{q}.$
 Now, since
$\mathfrak  {p}\in\bar{A^{*}}(\mathfrak  {p},N)$ and
$xN\neq N, $ it follows from Proposition \ref{1.4} that
 $\mathfrak  {p}\in\bar{A^{*}}(x\mathfrak  {p},N).$ Therefore, in view of
 \cite[Theorem 3.17]{ah} we have  $\mathfrak {p}\in E(x\mathfrak {p},N).$

  On other hand, since
  $x\not\in \bigcup_{\mathfrak{q}\in \mAss_{R}N}\mathfrak{q},$
  it is easily seen that
   $\mathfrak  {p}\not\in \mAss_{R}N/x\mathfrak{p}N,$ and so by  the assumption
   (i) and Lemma \ref{1.7} we have
  $\mathfrak  {p}\not\in E(x\mathfrak  {p},N),$ which is a contradiction. Hence,
  $\dim  N\leq 1.$ Now, we show that
   $\Ass_{R_{\mathfrak{p} }}N_{\mathfrak{p}}$ consists of a single prime ideal, for all
   $\mathfrak{p}\in \Supp(N).$ To do this, if
   $\dim     N=0,$ then $\dim    N_{\mathfrak{p}}=0.$ Hence
   $\Ass_{R_{\mathfrak{p} }}N_{\mathfrak{p}}=\lbrace{\mathfrak{p}} R_{\mathfrak{p} }\rbrace,$
   as required. Consequently, we have
   $\dim N_{\mathfrak{p}}=1.$ Now, if
    $\Ass_{R_{\mathfrak{p} }}N_{\mathfrak{p}}$ has at least two elements, then   in view of  Lemma
    \ref{1.5} there exists an ideal
    $\frak aR_{\mathfrak{p}}$ of
    $R_{\mathfrak{p}}$ such that
    $\mathfrak{p}R_{\mathfrak{p}}\in Q(\frak aR_{\mathfrak{p}},N_{\mathfrak{p}})$ but
    $\mathfrak{p}R_{\mathfrak{p}}\not\in \mAss_{R_{\mathfrak{p} }}N_{\mathfrak{p}}/\frak aR_{\mathfrak{p}}.$
   Therefore, in view of
   \cite[Lemma 3.2 and Theorem 3.17]{ah},
   $\mathfrak{p}\in E(\frak a,N)\setminus \mAss N/\frak aN,$
   which is a contradiction.

   In order to show the implication
  (ii)$\Longrightarrow$(i), in view of Lemma
   \ref{1.7} it is enough for us to show that
   $E(\frak b,N)=\mAss N/\frak bN.$ To this end, let
   $\mathfrak{p}\in E(\frak b,N).$ By virtue of
    \cite[Lemma 3.2]{ah}, we may assume that
    $(R,\mathfrak{p})$ is local.
   \\ Firstly, suppose $\dim N=0.$ Then it readily follows that
   $\frak p\in\mAss N/\frak bN$, as required. So we may assume that
   $\dim N=1.$ There are two cases to consider:\\

 \textbf{Case 1.}
 $\frak b\nsubseteq Z_R(N).$ Then
 $\grade(\frak b,N)>0.$ Since
 $\dim N=1,$ it follows that
 $\height_N\frak b=1,$ and so
 $\frak b+\Ann_RN$ is
 $\mathfrak{p}$-primary. Hence
 $\mathfrak{p}\in \mAss N/\frak bN,$ as required.\\

  \textbf{Case 2.} Now, suppose that
  $\frak b\subseteq Z_R(N).$ Then there exists
  $z\in \Ass_RN$  such that
  $\frak b\subseteq z.$ Since
  $\Ass_RN$ consists of a single prime ideal, so
  $\Ass_R N=\lbrace z\rbrace.$ Hence in view of
 \cite[Proposition 3.6]{ah},
 $\mathfrak{p}/z\in E(\frak b+z/z,R/z).$ Since
 $\frak b\subseteq z,$ it follows from \cite[Remark 2.3]{KR} that
 $ \mathfrak{p}=z,$  which is a contradiction, because
  $\dim N=1.$ Consequently,
 $\frak b\nsubseteq Z_R(N)$ and the claim holds.
\end{proof}


\begin{center}
{\bf Acknowledgments}
\end{center}
The authors would like to thank Professor Kamran Divaani-Aazar for reading of the first draft and valuable discussions.
Also, the authors would like to thank from the Institute for Research in Fundamental Sciences (IPM), for the financial support.


\end{document}